\documentclass[12pt]{iopart}
\usepackage{iopams}

\usepackage{amsthm}
\usepackage{amscd}
\usepackage{amssymb}
\usepackage{verbatim}
\usepackage{palatino}
\usepackage{dsfont,epsf}
\usepackage{psfrag,graphicx,epsf}

\usepackage{xy}
\xyoption{all}

\newtheorem{theorem}{Theorem}[section]
\newtheorem{lemma}[theorem]{Lemma}
\newtheorem{corollary}[theorem]{Corollary}
\newtheorem{proposition}[theorem]{Proposition}

\theoremstyle{definition}
\newtheorem{definition}[theorem]{Definition}
\theoremstyle{remark}

\newtheorem{example}[theorem]{Example}

\newcommand{\real}{{\mathbb R}}

\newcommand{\zed}{{\mathbb Z}}

\newcommand{\eg}{{\em e.g.}}

\newcommand{\del}{\partial}

\newcommand{\norm}[1]{\left\|{#1}\right\|}
\newcommand{\inv}{^{-1}}
\newcommand{\one}{{\mathbf{1}}}    
\newcommand{\Crit}{{\mathcal C}}        
\newcommand{\Index}{{\mathcal I}}       
\newcommand{\style}[1]{{\sc{#1}}}    
\newcommand{\crit}{{\mathcal C}}        
\newcommand{\Def}{{\sf Def}}            

\newcommand{\Bessel}{{\mathcal B}}       
\newcommand{\Fourier}{{\mathcal F}}       
\newcommand{\Family}{{\cal A}}       
\newcommand{\diam}{{\rm diam}}
\newcommand{\dchifloor}{{\lfloor d\chi\rfloor}}
\newcommand{\dchiceil}{{\lceil d\chi\rceil}}
\newcommand{\CF}{{\sf CF}}

\begin{document}

\title[Euler Transforms]{Euler-Bessel and Euler-Fourier Transforms}

\author{Robert Ghrist}
        \address{
        Departments of Mathematics and Electrical/Systems Engineering,
        University of Pennsylvania,
        Philadelphia PA, USA}
        \ead{ghrist@math.upenn.edu}

\author{Michael Robinson}
        \address{
        Department of Mathematics,
        University of Pennsylvania,
        Philadelphia PA, USA}
        \ead{robim@math.upenn.edu}
\maketitle

\begin{abstract}
We consider a topological integral transform of Bessel (concentric isospectral sets) type and Fourier (hyperplane isospectral sets) type, using the Euler characteristic as a measure. These transforms convert constructible $\zed$-valued functions to continuous $\real$-valued functions over a vector space. Core contributions include: the definition of the topological Bessel transform; a relationship in terms of the logarithmic blowup of the topological Fourier transform; and a novel Morse index formula for the transforms. We then apply the theory to problems of target reconstruction from enumerative sensor data, including localization and shape discrimination. This last application utilizes an extension of spatially variant apodization (SVA) to mitigate sidelobe phenomena.
\end{abstract}


\maketitle

\section{Introduction}
\label{sec_Intro}

Integral transforms are inherently geometric and, as such, are invaluable to applications in reconstruction. The vast literature on integral geometry possesses hints that many such integral transforms are at heart topological: note the constant refrain of Euler characteristic throughout integral-geometric results such as Crofton's Theorem or Hadwiger's Theorems (see, \eg, \cite{KR}). The parallel appearance of integral transforms in microlocal analysis --- especially from the sheaf-theoretic literature \cite{KS} --- confirms the role of topology in integral transforms.

This paper considers particular integral transforms designed to extract geometric features from topological data. The key technical tool involved is \style{Euler calculus} --- a simple integration theory using the (geometric or o-minimal) Euler characteristic as a measure (or valuation, to be precise). We define two types of Euler characteristic integral transforms: one, a generalization of the Fourier transform; the other, a generalization of the Hankel or Bessel transform. These generalizations are denoted \style{Euler-Fourier} and \style{Euler-Bessel} transforms, respectively. These transforms are novel, apart from the foreshadowing in the recent work on Euler integration \cite{BG:PNAS}.

The contributions of this paper are as follow:
\begin{enumerate}
\item
    Definitions of the Euler-Bessel and Euler-Fourier transforms on a normed (resp. inner-product) vector space;
\item
    Index theorems for both transforms which concentrate the transforms onto sets of critical points, and thereby reveal Morse-theoretic connections;
\item
    Applications of the Euler-Bessel transform to target localization and shape-discrimination problems; and
\item
    An extension of spatially variant apodization (SVA) to Euler-Bessel transforms, with applications to sidelobe mitigation  and to shape discrimination.
\end{enumerate}

\section{Background: Euler calculus}
\label{sec_euler}

The Euler calculus is an integral calculus based on the topological Euler characteristic. The reader may find a simple explanation of Euler calculus in \cite{BG} and more detailed treatments in \cite{Schapira:op,Viro}. For simplicity, we work in a fixed class of suitably ``tame'' or \style{definable} sets and mappings. In brief, an \style{o-minimal structure} is a collection of so-called definable subsets of Euclidean space satisfying certain axioms; a definable mapping between definable sets is a map whose graph is also a definable subset of the product: see \cite{vdD}. For concreteness, the reader may substitute ``semialgebraic'' or ``piecewise-linear'' or ``globally subanalytic'' for definable.  All definable sets are finitely decomposable into open simplices in a manner that makes the \style{Euler characteristic} invariant. Given a finite partition of a definable set $A$ into definable sets $\sigma_\alpha$ definably homeomorphic to open simplices,
\begin{equation}
    \chi(A) := \sum_\alpha(-1)^{\dim\sigma_\alpha} .
\end{equation}
This Euler characteristic also possesses a description in terms of alternating sums of (local) homology groups, yielding a topological invariance (up to homeomorphism for general definable spaces; up to homotopy for compact definable spaces).

The Euler characteristic is additive: $\chi(A\cup B)=\chi(A)+\chi(B)-\chi(A\cap B)$. Thus,  one defines a scale-invariant  ``measure'' $d\chi$ and an integral via characteristic functions: $\int\one_Ad\chi:=\chi(A)$ for $A$ definable. The collection of functions from a definable space $X$ to $\zed$ generated by finite (!) linear combinations of characteristic functions over compact definable sets is the group of \style{constructible} functions, $\CF(X)$. The \style{Euler integral} is the linear operator $\int_X\,\cdot\, d\chi:\CF(X)\to\zed$ taking the characteristic function $\one_\sigma$ of an open $k$-simplex $\sigma$ to $(-1)^k$. By additivity of $\chi$, the integral is well-defined \cite{Schapira:op,Viro}. There are several formulae for the computation (and numerical approximation) of integrals with respect to $d\chi$ \cite{BG}.

The integral with respect to $d\chi$ is well-defined for more general constructible functions taking values in a discrete subset of $\real$; however, continuously-varying integrands are problematic. A recent extension of the Euler integral to $\real$-valued definable functions uses a limiting process \cite{BG:PNAS}. Let $\Def(X)$ denote the definable functions from $X$ to $\real$ (those whose graphs in $X\times\real$ are definable sets). There is a pair of dual extensions of the Euler integral, $\dchifloor$ and $\dchiceil$, defined as follows:
\begin{equation}
\hspace{-0.5in}
    \int_X h\,\dchifloor
    =
    \lim_{n\to\infty}
    \frac{1}{n}\int_X \lfloor nh\rfloor d\chi ,
    \quad
    \int_X h\,\dchiceil
    =
    \lim_{n\to\infty}
    \frac{1}{n}\int_X \lceil nh\rceil d\chi .
\end{equation}
These limits exist and are well-defined, though {\em not equal} in general. The \style{triangulation theorem} for $\Def(X)$ \cite{vdD} states that to any $h\in\Def(X)$, there is a definable triangulation (a definable bijection to a disjoint union of open affine simplices in some Euclidean space) on which $h$ is affine on each open simplex. From this, one may reduce all questions about the integrals over $\Def(X)$ to questions of affine integrands over simplices, using the additivity of the integral. Using this reduction technique, one proves the following computational formulae \cite{BG:PNAS}:

\begin{theorem}
\label{thm:eulerintformulae}
For $h\in\Def(X)$,
\begin{eqnarray}
    \label{eq:eulerexcursion}
    \int_X h\,\dchifloor
    &=&
    \int_{s=0}^\infty \chi\{h\geq s\} - \chi\{h<-s\}\,ds
    \\
    \int_X h\,\dchiceil
    &=&
    \int_{s=0}^\infty \chi\{h>s\} - \chi\{h\leq-s\}\,ds
    .
\end{eqnarray}
\end{theorem}

This integral is coordinate-free, in the sense of being invariant under right-actions of homeomorphisms of $X$; however, the integral operators are {\em not linear}, nor, since $-\int h\dchifloor=\int h\dchiceil$, are they even homogeneous with respect to negative coefficients.  The compelling feature of the measures $\dchifloor$ and $\dchiceil$ is their relation to stratified Morse theory \cite{GM}.



Let $\Crit\subset X$ denote the set of critical points of $h$. For arbitrary $p\in \Crit$, the \style{co-index} of $p$, $\Index^*(p)$, is defined as
\begin{equation}
    \Index^*(p) = \lim_{\epsilon'\ll\epsilon\to 0^+}
        \chi\left(
        B_\epsilon(p)
        \cap
        \{h<h(p)+\epsilon'\}
        \right) ,
\end{equation}
where $B_\epsilon(p)$ denotes the closed ball in $X$ of radius $\epsilon$ about $p$. The dual \style{index} at $p$ is given by
\begin{equation}
    \Index_*(p) = \lim_{\epsilon'\ll\epsilon\to 0^+}
        \chi\left(
        B_\epsilon(p)
        \cap
        \{h>h(p)-\epsilon'\}
        \right) ,
\end{equation}

 \begin{theorem}[Theorem 4 of \cite{BG:PNAS}]
\label{thm:morseeuler}
If $h$ is continuous and definable on $X$, then:
\begin{equation}
\label{eq:stratmorse}
    \int_X h\,\dchifloor
    =
    \int_X h\, \Index^*\, d\chi ,
\quad
    \int_X h\,\dchiceil
    =
    \int_X h\, \Index_*\, d\chi.
\end{equation}
\end{theorem}

This has the effect of {\em concentrating} the measure $\dchifloor$ on the critical points of the distribution. In the case of $h$ a Morse function on an $n$-manifold $M$, for each critical point $p\in\Crit_x$, $\Index^*(p)=(-1)^{\dim M - \mu(p)}$ and $\Index_*(p)=(-1)^{\mu(p)}$, where $\mu(p)$ is the Morse index of $p$. Thus, if $h$ is a Morse function, then:
\begin{equation}
\label{eq:morse}
\hspace{-0.5in}
    \int_M h\,\dchifloor
    =
    \sum_{p\in\crit(h)}(-1)^{n-\mu(p)}h(p)
\quad
    \int_M h\,\dchiceil
    =
    \sum_{p\in\crit(h)}(-1)^{\mu(p)}h(p)
\end{equation}

\section{Definition: Euler-Bessel and Euler-Fourier transforms}
\label{sec_bessel}

There are a number of interesting integral transforms based on $d\chi$, including convolution and Radon-type transforms \cite{Brocker,Schapira:tom}. We introduce two Euler integral transforms on vector spaces for use in signal processing problems.

\subsection{Bessel}
For the Eulerian generalization of a Bessel transform, let $V$ denote a finite-dimensional vector space with (definable, continuous) norm $\norm{\cdot}$, and let $B_r(x)$ denote the compact ball of points $\{y \, : \, \norm{y-x}\leq r\}$. Recall that $\CF$ denotes compactly-supported definable integer-valued functions.

\begin{definition}
For $h\in \CF(V)$ define the \style{Bessel transform} of $h$ via
\begin{equation}
    \Bessel h(x) = \int_0^\infty\int_{\partial B_r(x)}h\,d\chi\,dr .
\end{equation}
\end{definition}

This transform Euler-integrates $h$ over the concentric spheres at $x$ of radius $r$, and Lebesgue-integrates these spherical Euler integrals with respect to $r$. For the Euclidean norm, these isospectral sets are round spheres. Given our convention that $\CF(V)$ consists of compactly supported functions, $\Bessel:\CF(V)\to\zed$ is well-defined using standard o-minimal techniques (the Conic Theorem \cite{vdD}).


%

\subsection{Fourier}

There is a similar integral transform that is best thought of as a topological version of the Fourier transform. This is a global version of the microlocal Fourier-Sato transform on the sheaf $\CF(V)$ \cite{KS}. For this transform, an inner product on $V$ must be specified. The Fourier transform takes as its argument a covector $\xi\in V^*$.

\begin{definition}
For $h\in \CF(V)$ define the \style{Fourier transform} of $h$ in the direction $\xi\in V^*$ via
\begin{equation}
    \Fourier h(\xi) = \int_0^\infty\int_{\xi\inv(r)}h\,d\chi\,dr .
\end{equation}
\end{definition}

\begin{example}
\label{ex:fourierwidth}
For $A$ a compact convex subset of $\real^n$ and $\norm{\xi}=1$, $(\Fourier\one_A)(\xi)$ equals the projected length of $A$ along the $\xi$-axis.
\end{example}

The Bessel transform can be seen as a Fourier transform of the log-blowup. This perspective leads to results like the following.

\begin{proposition}
\label{prop:fourierbessel}
The Bessel transform along an asymptotic ray is the Fourier transform along the ray's direction: for $h\in\CF(V)$ and $x\neq 0\in V$,
\begin{equation}
    \lim_{\lambda\to\infty}(\Bessel h)(\lambda x) = (\Fourier h)\left(\frac{x^*}{\norm{x^*}}\right) .
\end{equation}
where $x^*$ is the dual covector.
\end{proposition}
\begin{proof}
The isospectral sets restricted to the (compact) support of $h$ converge in the limit and the scalings are identical.
\end{proof}

While the Fourier transform obviously measures a ``width'' associated to a constructible function, the geometric interpretation of the Bessel transform is more involved. The next section explores this geometric content via index theory.

\section{Computation: Index-theoretic formulae}
\label{sec_index}

The principal results of this paper is are index formulae for the Euler-Bessel and Euler-Fourier transforms that reduce the integrals to critical values.


\begin{lemma}
\label{lem:cone}
For $A\subset V$ the closure of a open subset of $V$, star-convex with respect to $x\in A$,
\begin{equation}
    \Bessel\one_A(x) = \int_{\del A}d_x\,\dchifloor ,
\end{equation}
where $d_x$ is the distance-to-$x$ function $d_x:V\to\real^+$.
\end{lemma}
\begin{proof}
Consider the logarithmic blowup taking $V-\{x\}\cong S^{n-1}\times\real^+$, with the second coordinate being $\norm{y-x}$. The level sets of the $\real^+$ coordinate of the blowup are precisely the isospectral sets of $\Bessel(x)$. The induced height function $d_x:\del A\to\real^+$ is well-defined on the unit tangent sphere of $x$, $\del_A\cong S^{n-1}$, since $A$ is star-convex with respect to $x$. By definition,
\[
            \Bessel h (x) = \int_0^\infty \chi(A\cap \del B_r(x)) ds .
\]
For $A$ star-convex and top-dimensional, $A\cap \del B_r(x)$ is homeomorphic to $\del A\cap\{d_x\geq r\}$. By Equation (\ref{eq:eulerexcursion}),
\[
            \Bessel h (x) = \int_0^\infty \chi(\del A\cap\{d_x\geq r\})dr = \int_{\del A} d_x\,\dchifloor .
\]
\end{proof}

This theorem is a manifestation of Stokes' Theorem: the integral of the distance over $\del A$ equals the integral of the `derivative' of distance over $A$. For non-star-convex domains, it is necessary to break up the boundary into positively and negatively oriented pieces. These orientations implicate $\dchifloor$ and $\dchiceil$ respectively.

\begin{theorem}
\label{thm:besseldchi}
For $A\subset V$ the closure of a definable bounded open subset of $V$ and $x\in V$, decompose $\del A$ into $\del A=\del_x^+A\cup\del_x^-A$, where $\del_x^\pm A$ are the (closure of) subsets of $\del A$ on which the outward-pointing halfspaces  contain (for $\del_x^-$) or, respectively, do not contain (for $\del_x^+$) $x$. Then,
\begin{eqnarray}
\label{eq:besselindex}
    \Bessel\one_A(x)
            &=& \int_{\del_x^+A} d_x\, \dchifloor - \int_{\del_x^-A} d_x\, \dchiceil \\
            &=& \int_{\Crit_x\cap\del_x^+A} d_x\, \Index^*\, d\chi - \int_{\Crit_x\cap\del_x^-A} d_x\, \Index_*\, d\chi .
\end{eqnarray}
where $\Crit_x$ denotes the critical points of $d_x:\del A\to[0,\infty)$.
\end{theorem}
\begin{proof}
Assume, for simplicity, that $A$ is the closure of the difference of $C^+_x$, the cone at $x$ over $A^+_x$, and $C^-_x$, the cone over $A^-_x$. (The case of multiple cones follow by induction.) These cones, being star-convex with respect to $x$, admit analysis as per Lemma \ref{lem:cone}. The crucial observation is that, by additivity of $\chi$,
\[
        \chi(\del B_r(x) \cap A) = \chi(\del C^+_x\cap\{d_x\geq r\}) - \chi(\del C^-_x\cap\{d_x>r\}) .
\]
Integrating both sides with respect to $dr$ and invoking Theorem \ref{thm:eulerintformulae} gives
\[
        \Bessel\one_A(x) = \int_{\del C^+_x}d_x\,\dchifloor - \int_{\del C^-_x}d_x\,\dchiceil .
\]
By Theorem \ref{thm:morseeuler}, this reduces to an integral over the critical sets of $d_x$. The only critical point of $d_x$ on $C^+_x-\del A$ or $C^-_x-\del A$ is $x$ itself, on which the integrand $d_x$ takes the value $0$ and does not contribute to the integral. Therefore the integrals over the cone boundaries may be restricted to $\del^+A$ and $\del^-A$ respectively. The index-theoretic result follows from Theorem \ref{thm:morseeuler}.
\end{proof}

In even dimensions, the $\dchifloor$-vs-$\dchiceil$ dichotomy dissolves:

\begin{corollary}
\label{cor:evendim}
For $\dim V$ even and $A$ the closure of a bounded definable open set,
\begin{equation}
\label{eq:evendim}
    \Bessel\one_A(x) = \int_{\del A} d_x\, \dchifloor  = \int_{\Crit_x} d_x\,\Index_*\,d\chi .
\end{equation}
\end{corollary}
\begin{proof}
For $\dim V$ even, $\dim\del A$ is odd. Equation {\bf [18]} of \cite{BG:PNAS} implies that on an odd-dimensional manifold, $\int \dchiceil=-\int \dchifloor$. Equation (\ref{eq:besselindex}) completes the proof.
\end{proof}


Given the index theorem for the Euler-Bessel transform, that for the Euler-Fourier is a trivial modification that generalizes Example \ref{ex:fourierwidth}.

\begin{theorem}
\label{thm:fourierdchi}
For $A\subset V$ the closure of a definable bounded open subset of $V$ and $\xi\in V^*-\{0\}$, decompose $\del A$ into $\del A=\del_x^+A\cup\del_x^-A$, where $\del_x^\pm A$ are the (closure of) subsets of $\del A$ on which $\xi$ points out of ($\del^+$) or into ($\del^-$) $A$. Then,
\begin{eqnarray}
\label{eq:fourierindex}
    \Fourier\one_A(\xi)
            &=& \int_{\del_{\xi}^+A} \xi\, \dchifloor - \int_{\del_{\xi}^-A} \xi\, \dchiceil \\
            &=& \int_{\Crit_\xi\cap\del_{\xi}^+A} \xi\, \Index^*\, d\chi - \int_{\Crit_\xi\cap\del_\xi^-A} \xi\, \Index_*\, d\chi .
\end{eqnarray}
where $\Crit_\xi$ denotes the critical points of $\xi:\del A\to[0,\infty)$. For $\dim V$ even, this becomes:
\begin{equation}
    \Fourier\one_A(\xi) = \int_{\del A} \xi\, \dchifloor  = \int_{\Crit_\xi} \xi\,\Index_*\,d\chi .
\end{equation}
\end{theorem}

The proof follows that of Theorem \ref{thm:besseldchi} and is an exercise. Figure \ref{fig:index} gives a simple example of the Bessel and Fourier index theorems in $\real^2$.

\begin{figure}
\begin{center}
\psfragscanon
\psfrag{-}[][]{$-$}
\psfrag{+}[][]{$+$}
\psfrag{x}[][]{$x$}
\psfrag{X}[][]{$\xi$}
\psfrag{A}[][]{$A$}
\psfrag{p}[][]{$\del_x^+A$}
\psfrag{m}[][]{$\del_x^-A$}
\psfrag{P}[][]{$\del_\xi^+A$}
\psfrag{M}[][]{$\del_\xi^-A$}
\includegraphics[width=5in]{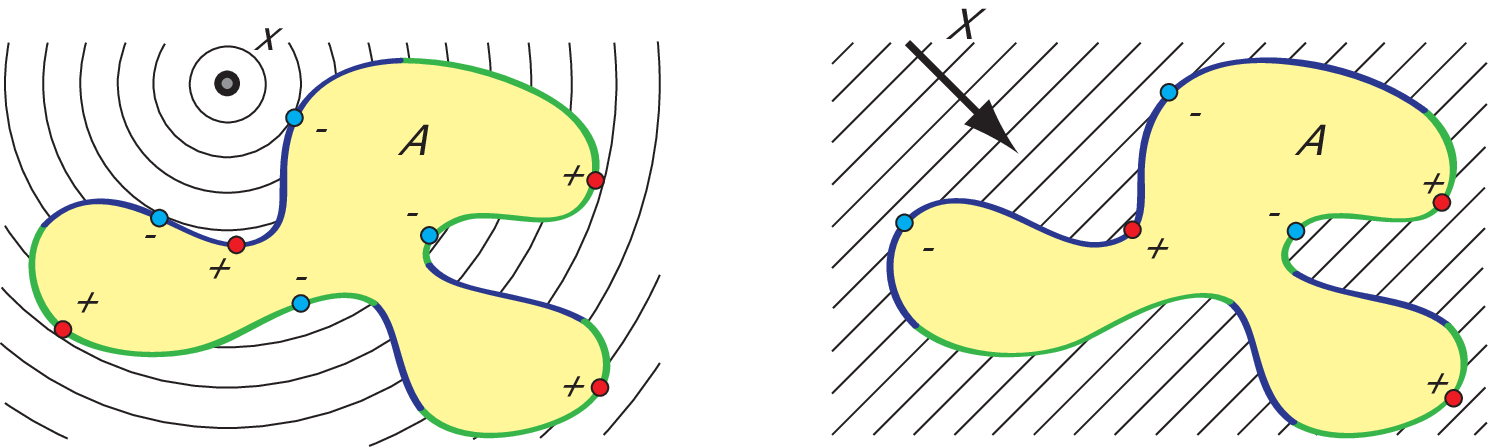}
\caption{The index formula for $\Bessel$ [left] and $\Fourier$ [right] applied to $\one_A$ localizes the transform to (topological or smooth) tangencies of the isospectral sets with $\del A$.}
\label{fig:index}
\end{center}
\end{figure}

By linearity of $\Bessel$ and $\Fourier$ over $\CF(V)$, one derives index formulae for integrands in $\CF(V)$ expressible as a linear combination of $\one_{A_i}$ for $A_i$ the closure of definable bounded open sets. For a set $A$ which is not of dimension $\dim V$, it is still possible to apply the index formula by means of a limiting process on compact tubular neighborhoods of $A$.


The remainder of this paper explores applications of the Euler-Bessel transform to signal processing problems involving target detection, localization, and discrimination.

\section{Application: Target localization}
\label{sec_localization}

Applications of Euler-type integral transforms are naturally made in the context of (reasonably dense) sensor networks. As in \cite{BG}, we consider the setting in which a finite number of targets reside in a field of sensors whose locations are parameterized by a vector space $V$. Each target has a corresponding \style{support} --- the subset of $V$ on which a sensor senses the target, albeit without information of the target's range, bearing, or identity. The resulting sensor counting function $h\in\CF(V)$ contains highly redundant but informative structure. The Euler integral can be used to extract from $h$ the number of targets \cite{BG}; the Euler-Bessel transform can be used to localize the targets in the case of convex target supports. Assume for example that all target supports are equal to a round ball: all sensors within some fixed distance of the target detect it. The Bessel transform of the sensor field reveals the exact location of the target.

\begin{proposition}
\label{prop:monotone}
For $A=B_R(p)$ a compact ball about $p\in\real^{2n}$, the Bessel transform $\Bessel\one_A$ is a nondecreasing function of the distance to $p$, having unique zero at $p$.
\end{proposition}
\begin{proof}
Convexity of balls and Corollary \ref{cor:evendim} implies that
\[
        \Bessel\one_A(x) = \int_{\del A} d_x\,\dchifloor = \max_{\del A}d_x - \min_{\del A}d_x ,
\]
which equals $\diam A=2R$ for $x\not\in A$ and is monotone in distance-to-$p$ within $A$.
\end{proof}

Note that Proposition \ref{prop:monotone} fails in odd dimensions; the Bessel transform of a ball in $\real^{2n+1}$ is constant, and $\Bessel$ obscures all information. However, for even dimensions, Proposition \ref{prop:monotone} provides a basis for target localization. For targets with convex supports (regions detected by counting sensors), the local minima of the Euler-Bessel transform can reveal target locations: see Fig. \ref{fig:besselloc}[left] for an example. Note that in this example, not all local minima are target centers: interference creates ghost minima. However, given $h\in\CF(V)$, the integral $int_Vh\,d\chi$ determines the number of targets. This provides a guide as to how many of the deepest local minima to interrogate.

\begin{figure}[hbt]
\begin{center}
\includegraphics[width=2.0in,height=2.0in]{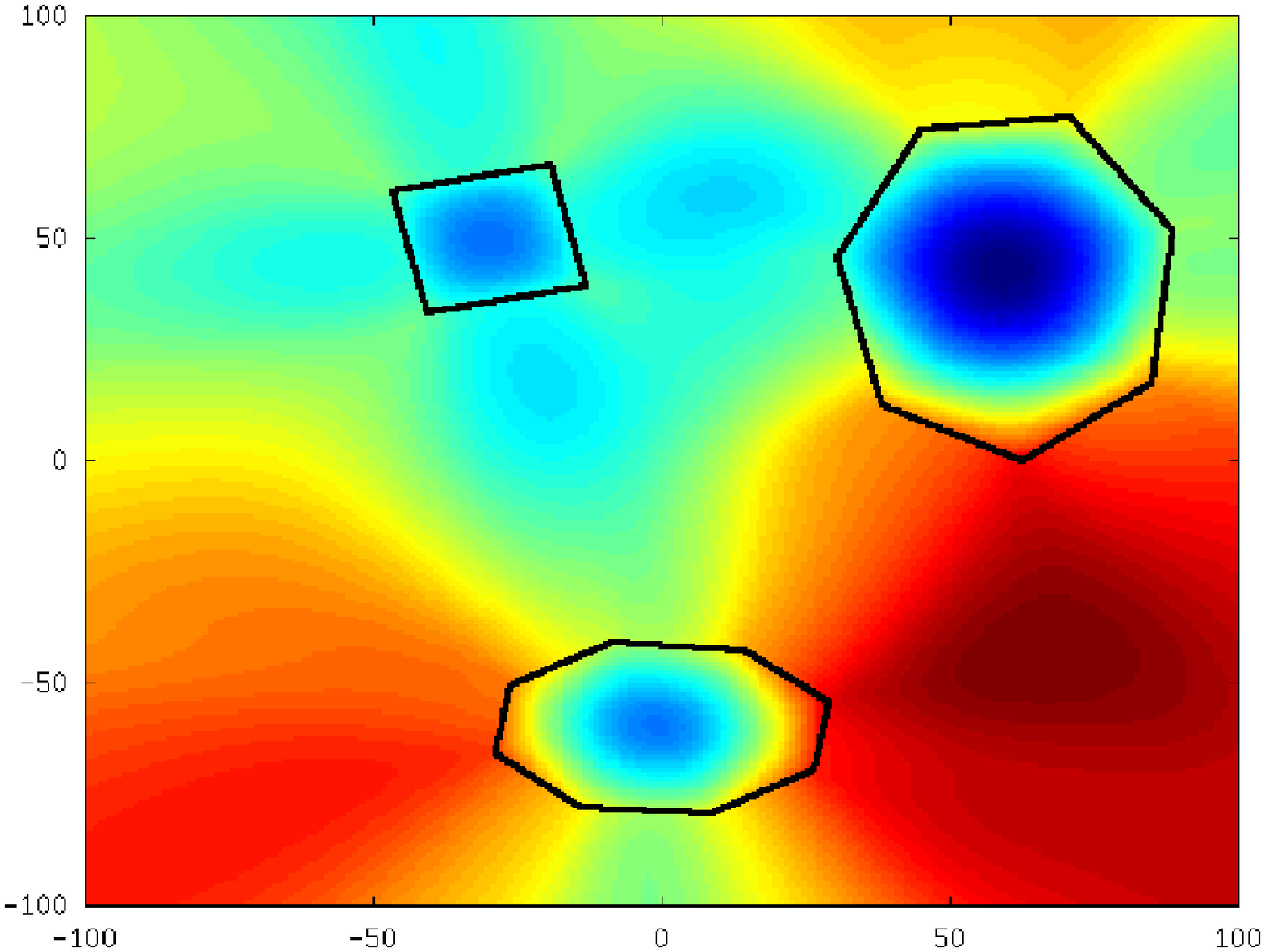}
\includegraphics[width=2.0in,height=2.0in]{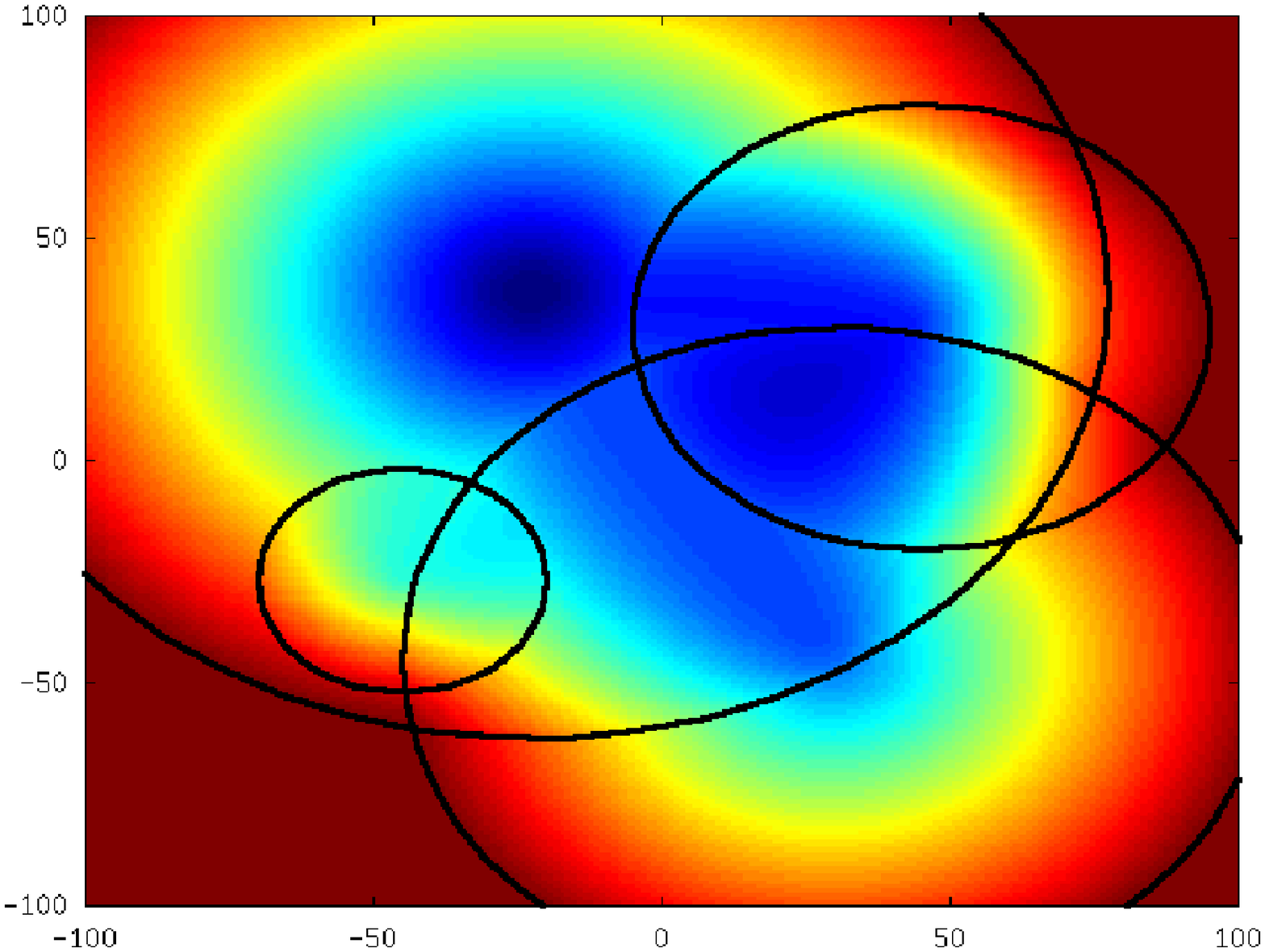}
\caption{The Euler-Bessel transform of a collection of convex targets [left] has local minima at the target centers. However, too much interference between targets obscures target centers [right].}
\label{fig:besselloc}
\end{center}
\end{figure}

There are significant limitations to superposition by linearity for this application. When targets are nearby or overlapping, their individual transforms will have overlapping sidelobes, which results in uncertainty when the transform is being used for localization.  A typical example of this difficulty (present even in the setting of round targets) is shown in Figure \ref{fig:besselloc}[right].

\section{Spatially variant apodization and sidelobes}
\label{sec_lobe}

Figure \ref{fig:besselloc} reveals the prevalence of \style{sidelobes} in the application of the Euler-Bessel transform --- regions of ``energy leakage'' in the transform --- much the same as occurs in Lebesgue-theoretic integral transforms. These are, in general, disruptive, especially in the context of several targets, where multiple side lobe interferences can create ``ghost'' images. Such problems are prevalent in traditional radar image processing, and their wealth of available perspectives in mitigating unwanted sidelobes provides a base of intuition for dealing with the present context.

One obvious method for modifying the output of Euler-Bessel transform is to change the isospectral contours of integration , as regulated by the norm $\norm{\cdot}$.  For example, one can contrast the circular (or $\ell_2$) Euler-Bessel transform with its square (or $\ell_\infty$) variant. The resulting outputs of norm-varied transforms can have incisive characteristics in some circumstances: see \S\ref{sec_waveforms}. However, it may the case that no single norm is optimal for a given input, especially if it consists of multiple targets with different characteristics.

%
%
%

We propose the adaptation and refinement of one tool of widespread use in traditional radar processing. This usually goes under the name of SVA: \style{spatially variant apodization} (see \cite{SDF} for the original implementation; an updated discussion is in \cite{Carrara}). Though there are many heuristic implementations of SVA, the core concept behind the method involves using a parameterized family of kernels and optimizing the transform pointwise with respect to this family. The family of kernels is designed so as to reduce as much as possible the magnitude of the sidelobe phenomena, while preserving as much as possible the primary lobe.

For the setting of the Euler-Bessel transform, we propose the following. Consider a parameterized family $\Family$ of norms $\norm{\cdot}_\alpha$, $\alpha\in\Family$. The SVA Euler-Bessel transform is the pointwise infimum of transforms over $\Family$.
\begin{equation}
\left(\Bessel_{SV\!A} h \right) (x)=\inf_{\alpha\in\Family} \int_0^\infty \int_{B_{r,\alpha}(x)} h\, d\chi \, dr,
\end{equation}
where $B_{r,\alpha}(x)$ is the radius $r$ ball about $x$ in the $\alpha$-norm. For example, $\Family$ could describe a cyclic family of rotated $\ell_\infty$ norms.  As shown in Figure \ref{sva_rsq_two_fig}, SVA with this rotated family eliminates the sidelobes from the transform of a square rotated by an unknown amount, while preserving the response of another nearby target.

\begin{figure}
\begin{center}
\includegraphics[width=2.0in,height=2.0in]{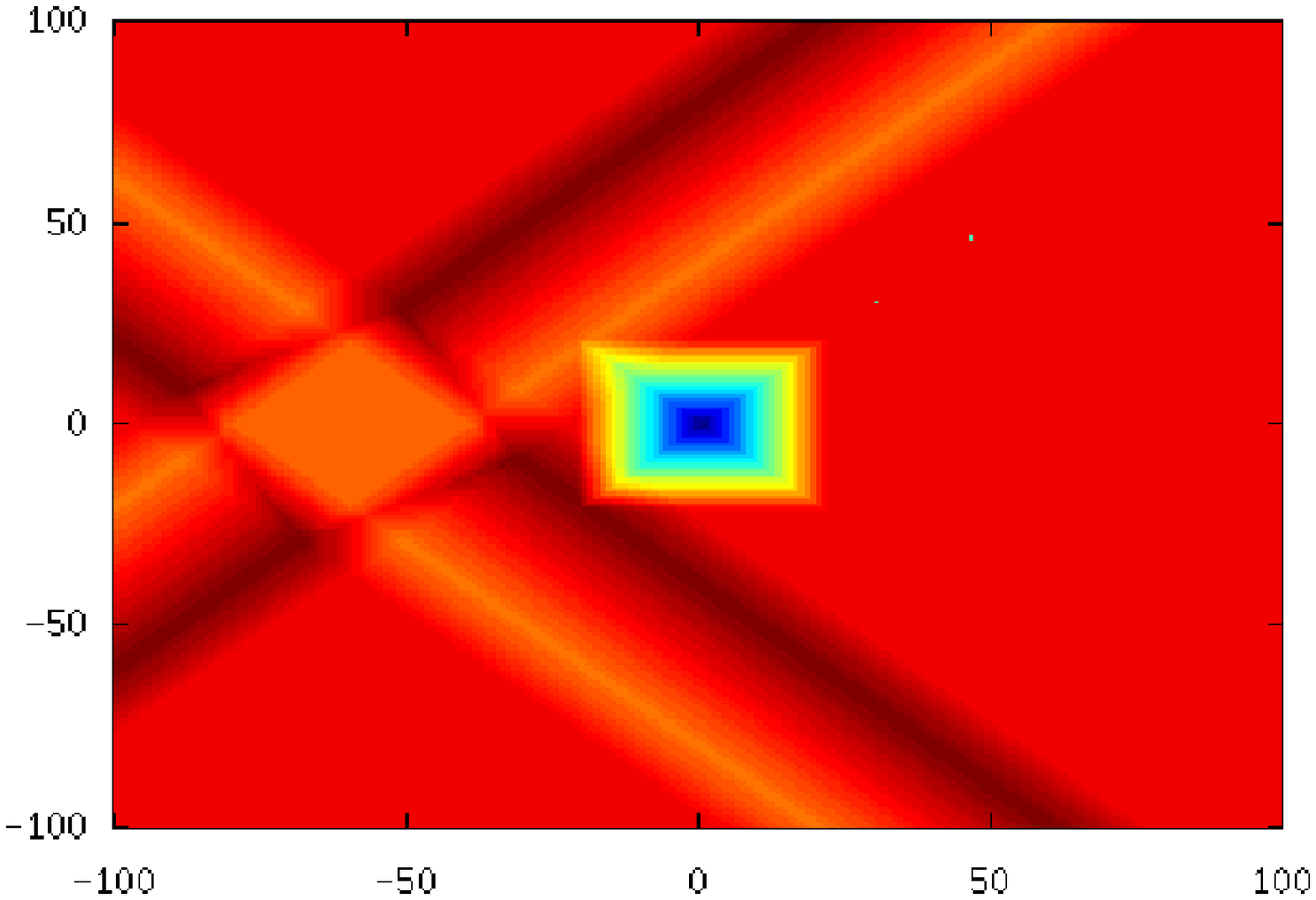}
\includegraphics[width=2.0in,height=2.0in]{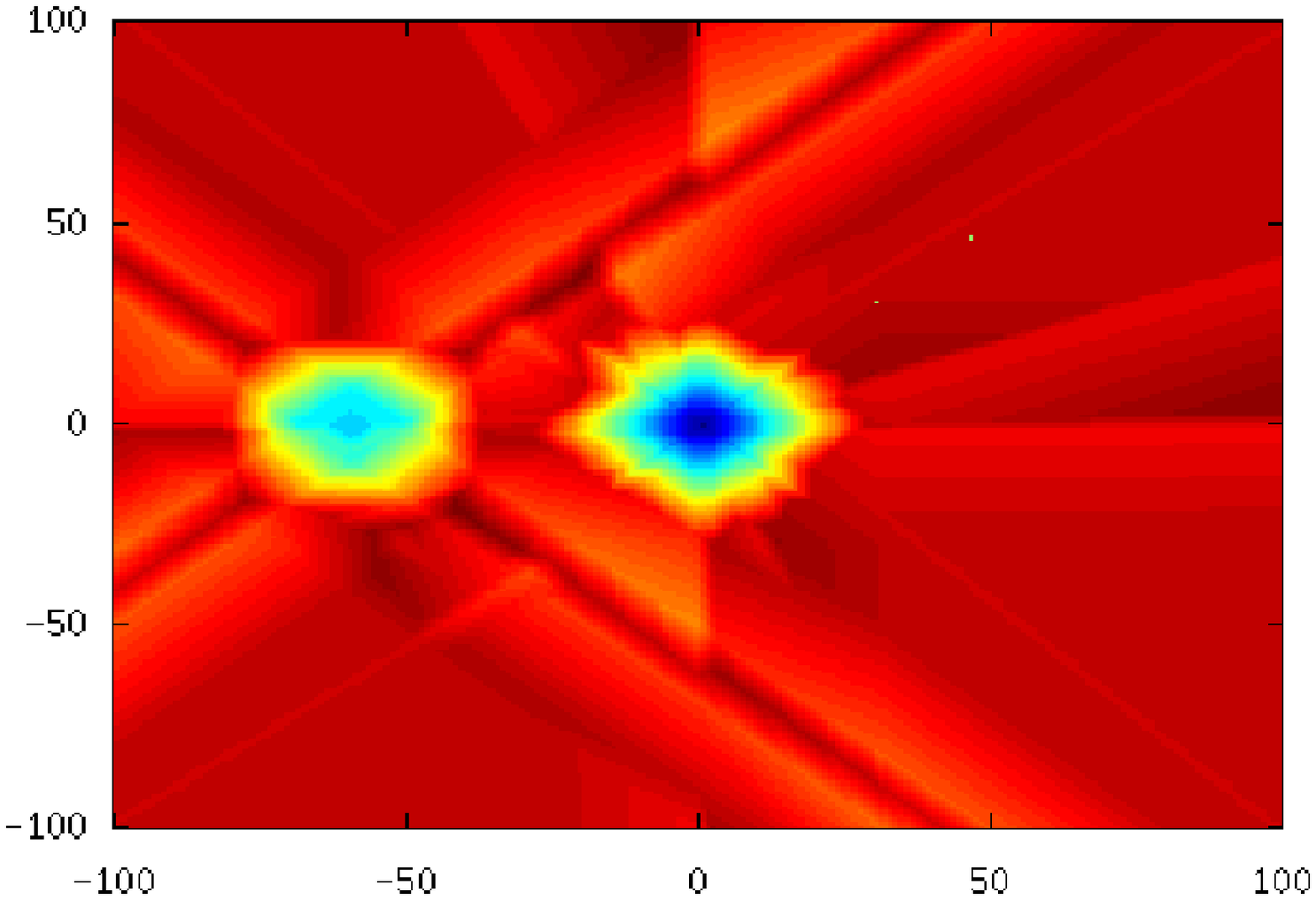}
\caption{Before (left) and after (right) application of SVA to the $\ell_\infty$ Euler-Bessel transform of a pair of disjoint square targets.}
\label{sva_rsq_two_fig}
\end{center}
\end{figure}

As in the case of traditional radar processing, there is benefit to SVA implementation even when the transform contour is not similar to the expected target shape.  For instance, $\ell_\infty$ contours result in strong sidelobes in the transform of a hexagon (Figure \ref{sva_hex_fig}[left]).  Application of rotated SVA (Figure \ref{sva_hex_fig}[right]) does not eliminate the sidelobes, but does dramatically reduce their magnitude.

\begin{figure}
\begin{center}
\includegraphics[width=2.0in,height=2.0in]{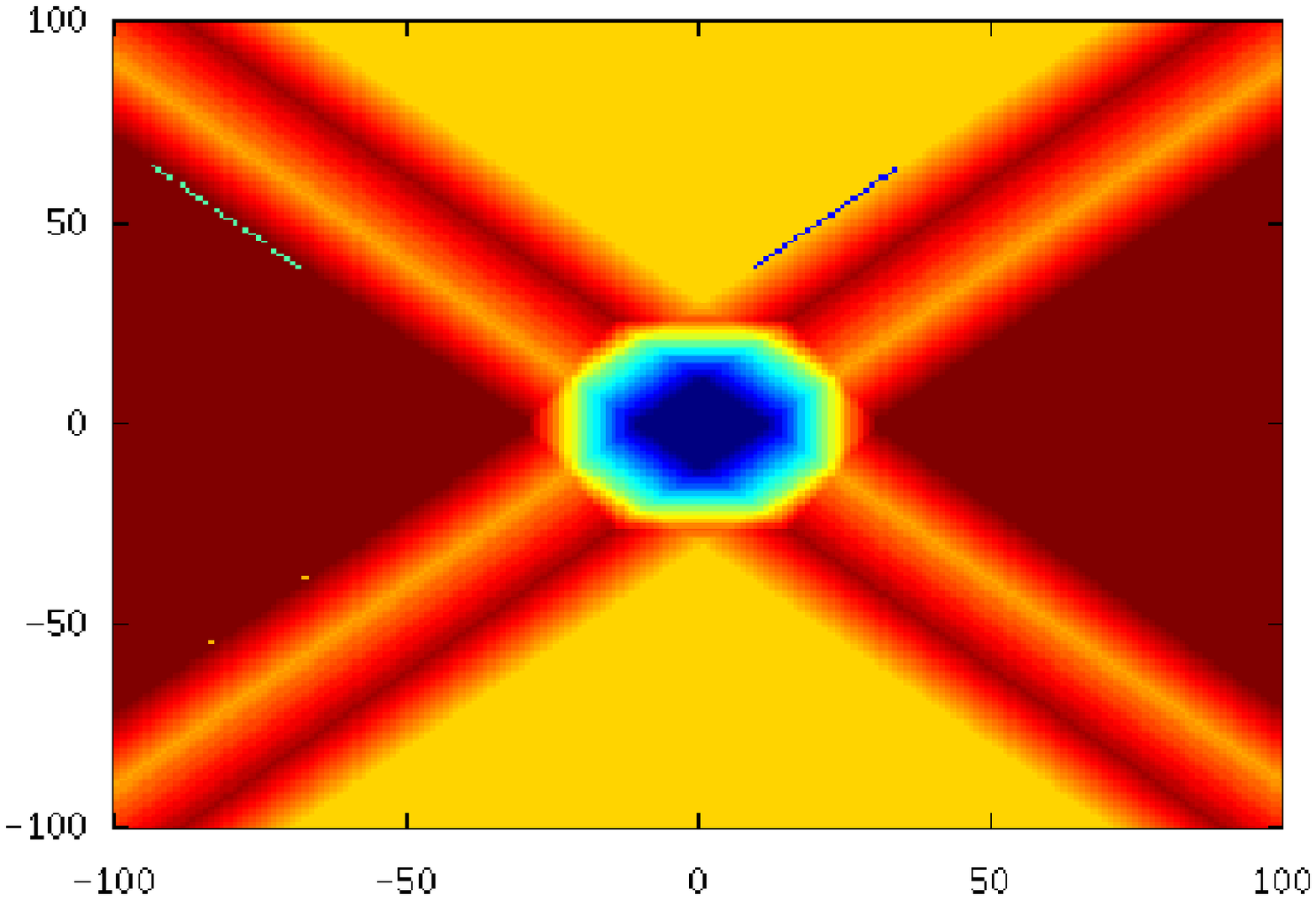}
\includegraphics[width=2.0in,height=2.0in]{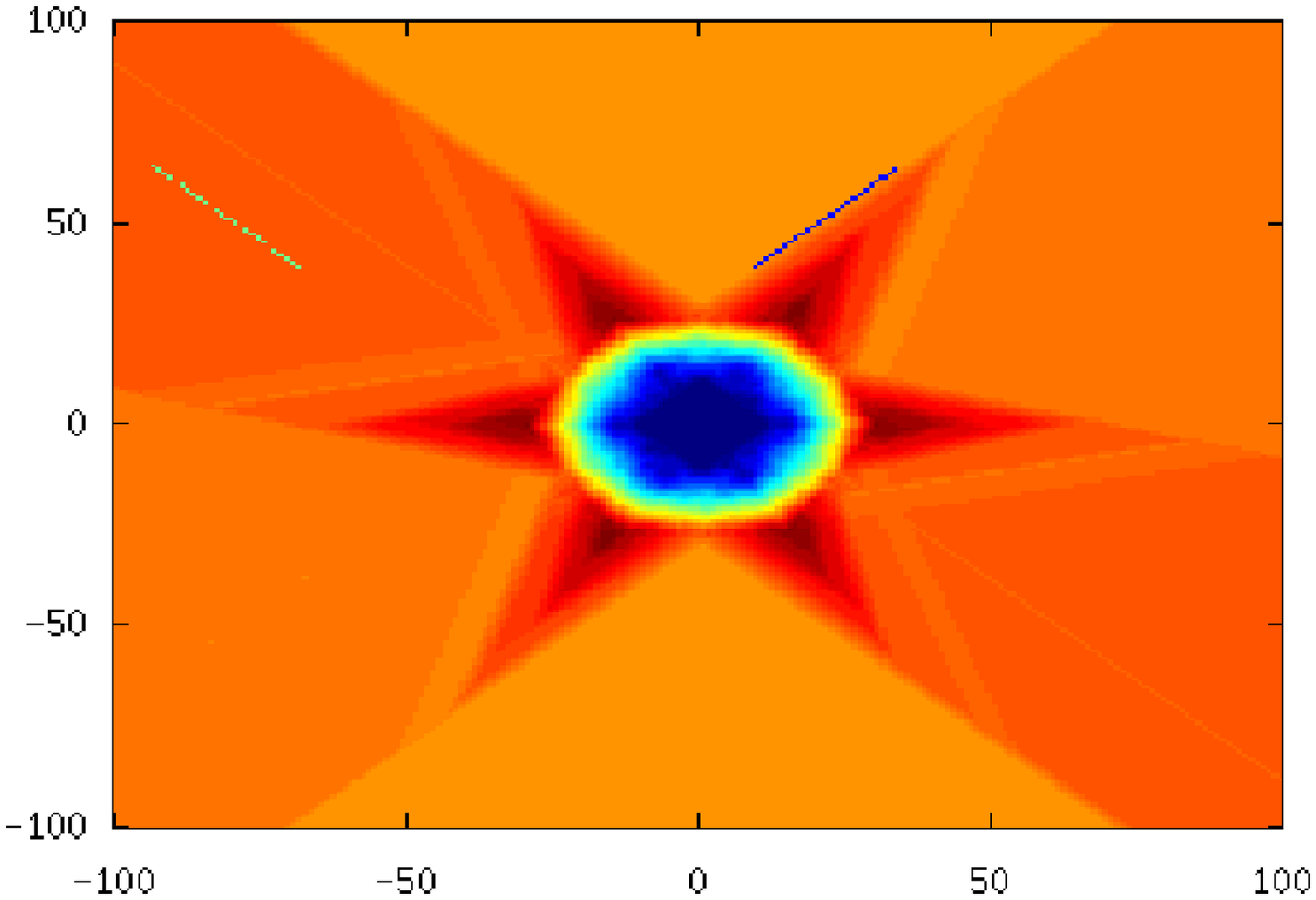}
\caption{Before (left) and after (right) application of SVA to the $\ell_\infty$ Euler-Bessel transform of a hexagonal target support. Use of SVA greatly mitigates sidelobes.}
\label{sva_hex_fig}
\end{center}
\end{figure}



%

\section{Application: Waveforms and shape discrimination}
\label{sec_waveforms}

Motivated by matched-filter processing, we apply Euler-Bessel transforms to perform geometric discrimination from enumerative data.  Since the Euler-Bessel transform of a target has minimal sidelobes when the isospectral contours agree with the target shape, it can be used to construct a \style{shape filter}. To test this, a simulation was run on a function $h=\sum_i\one_{A_i}$ for $A_i$ a collection of polygonal domains in $\real^2$.  The results are shown in Figures \ref{fig:shapefinder1} through \ref{fig:shapefinder4}. In these, $A_1$ (upper left) is a hexagon; $A_2$ (upper right) is a rotated square; $A_3$ (lower left) is a round disk; and $A_4$ (lower right) is an axis-aligned square. The sizes of these support regions were varied, to the point of significant overlap and interference.

\begin{figure}
\begin{center}
\includegraphics[width=1.8in,height=1.8in]{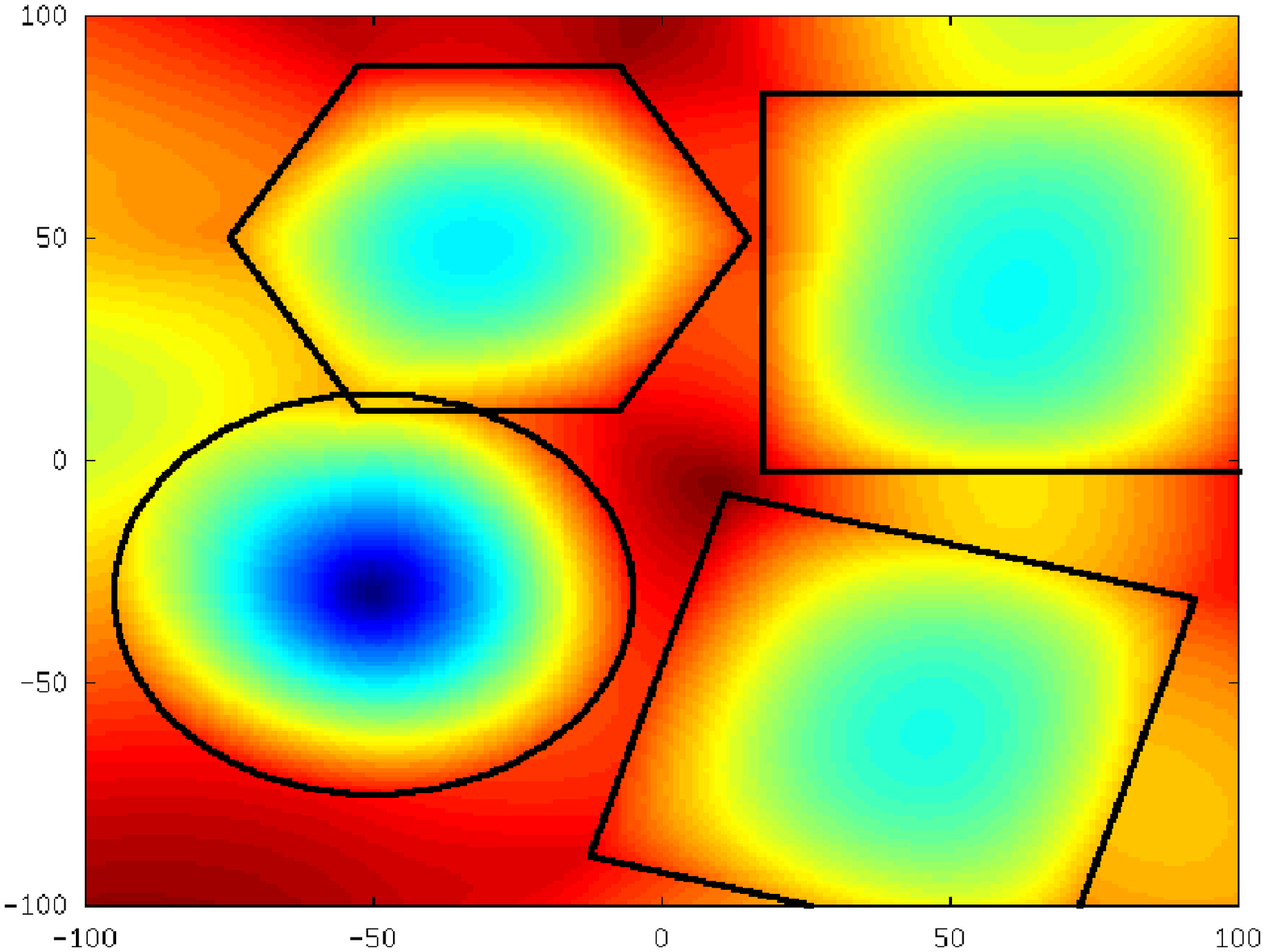}
\includegraphics[width=1.8in,height=1.8in]{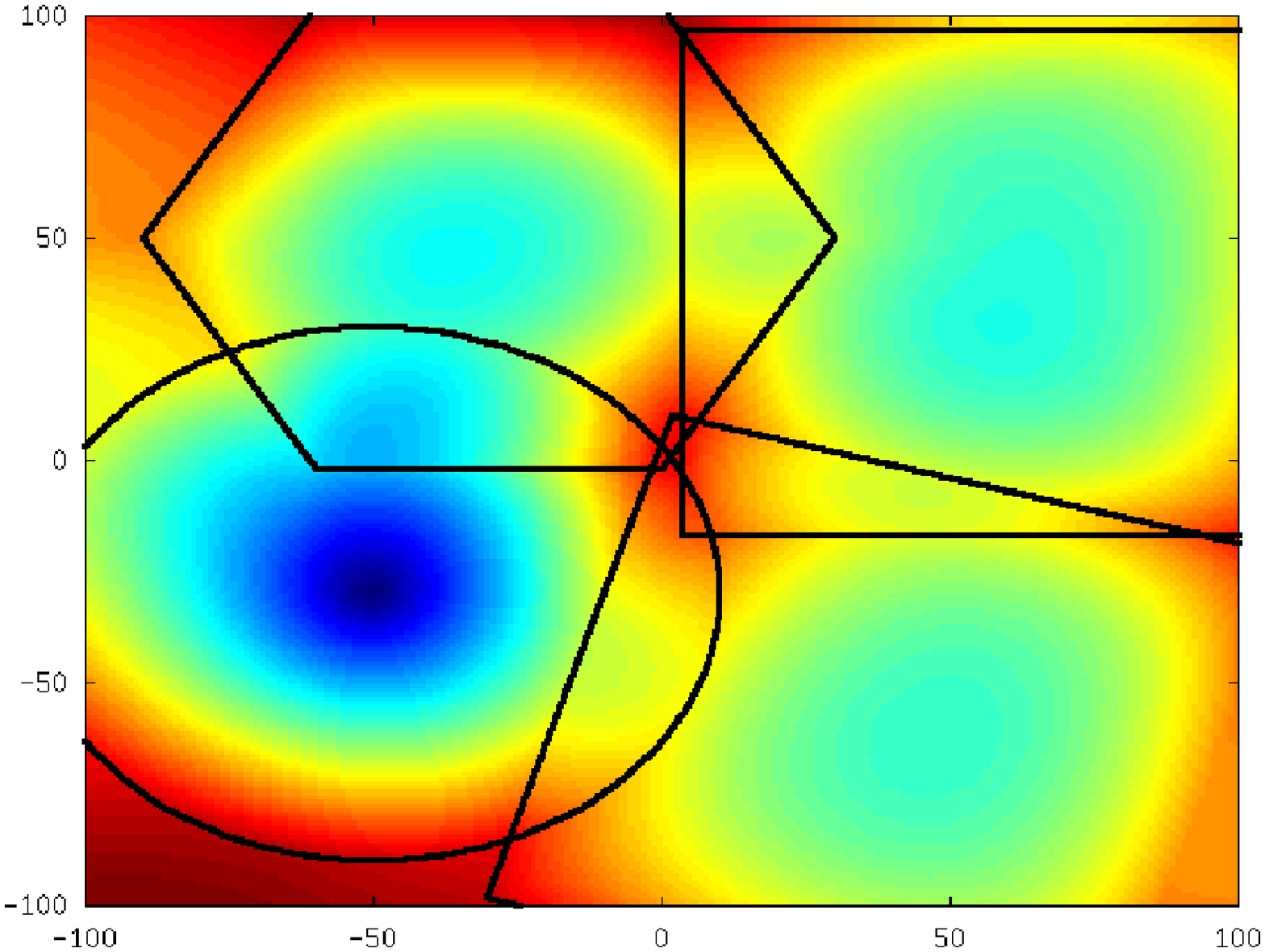}
\includegraphics[width=1.8in,height=1.8in]{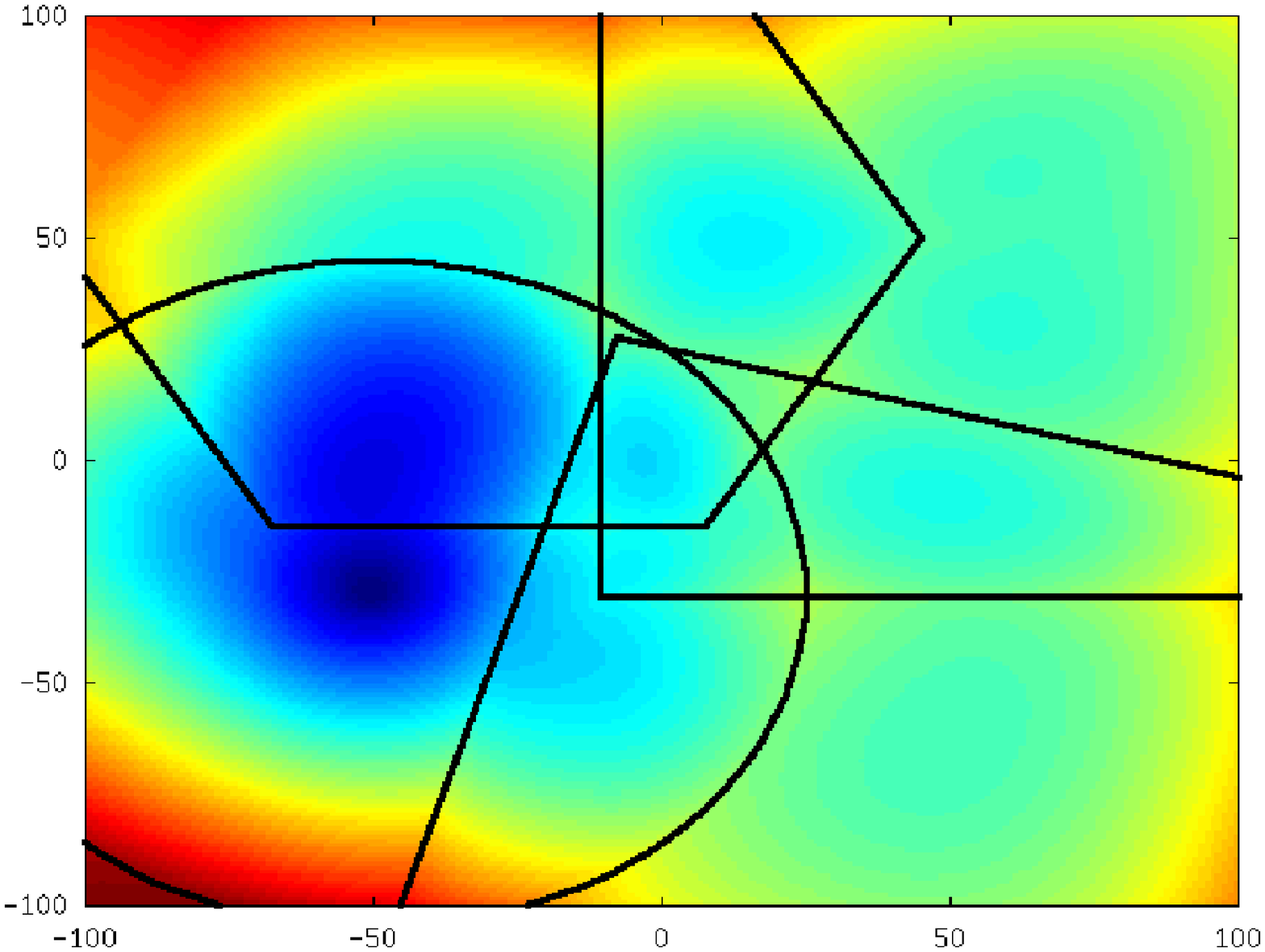}
\caption{Circular ($\ell_2$) Euler-Bessel transforms of a collection of interfering target supports reveals the location of the round target.}
\label{fig:shapefinder1}
\end{center}
\end{figure}

\begin{figure}
\begin{center}
\includegraphics[width=1.8in,height=1.8in]{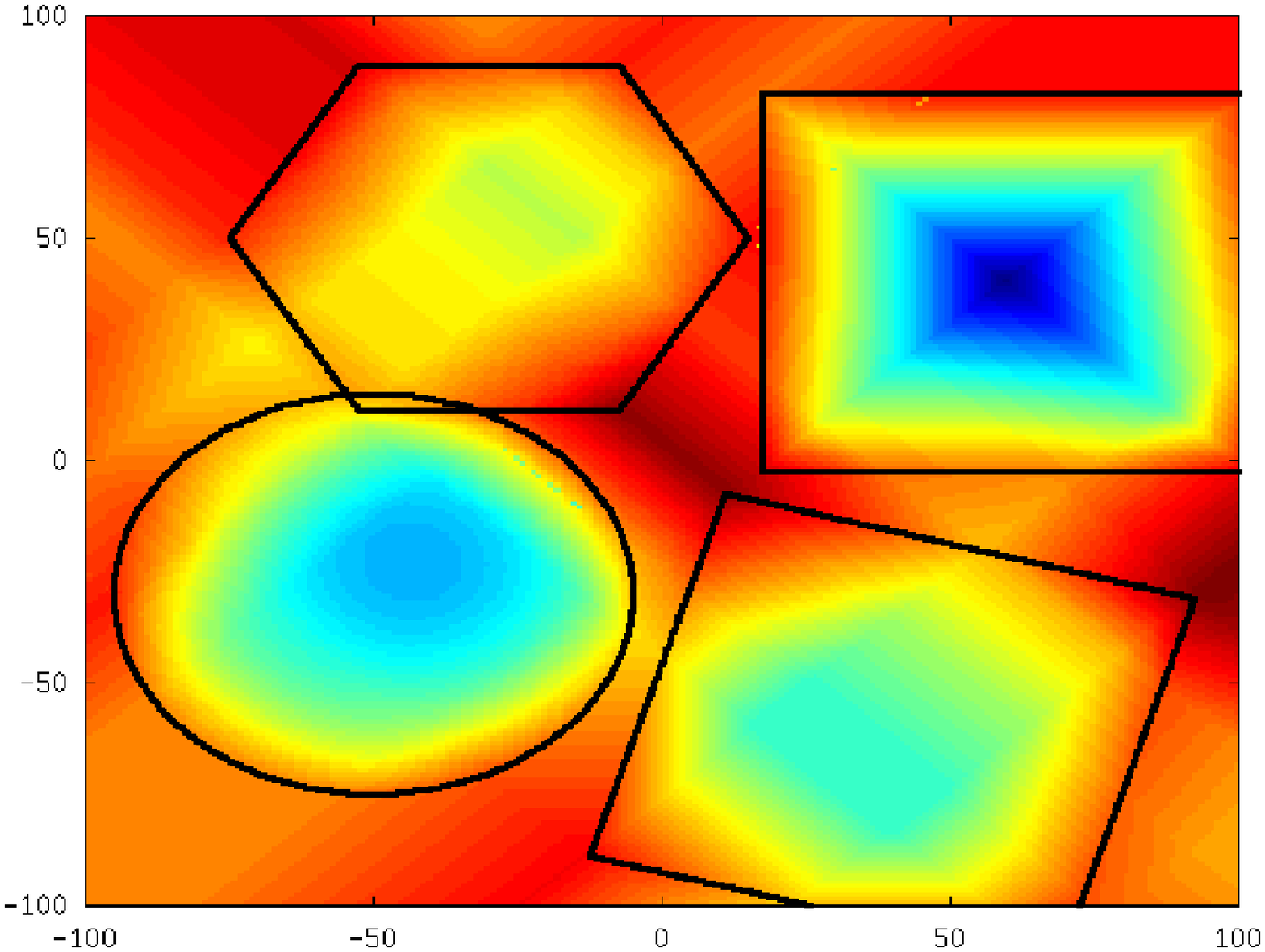}
\includegraphics[width=1.8in,height=1.8in]{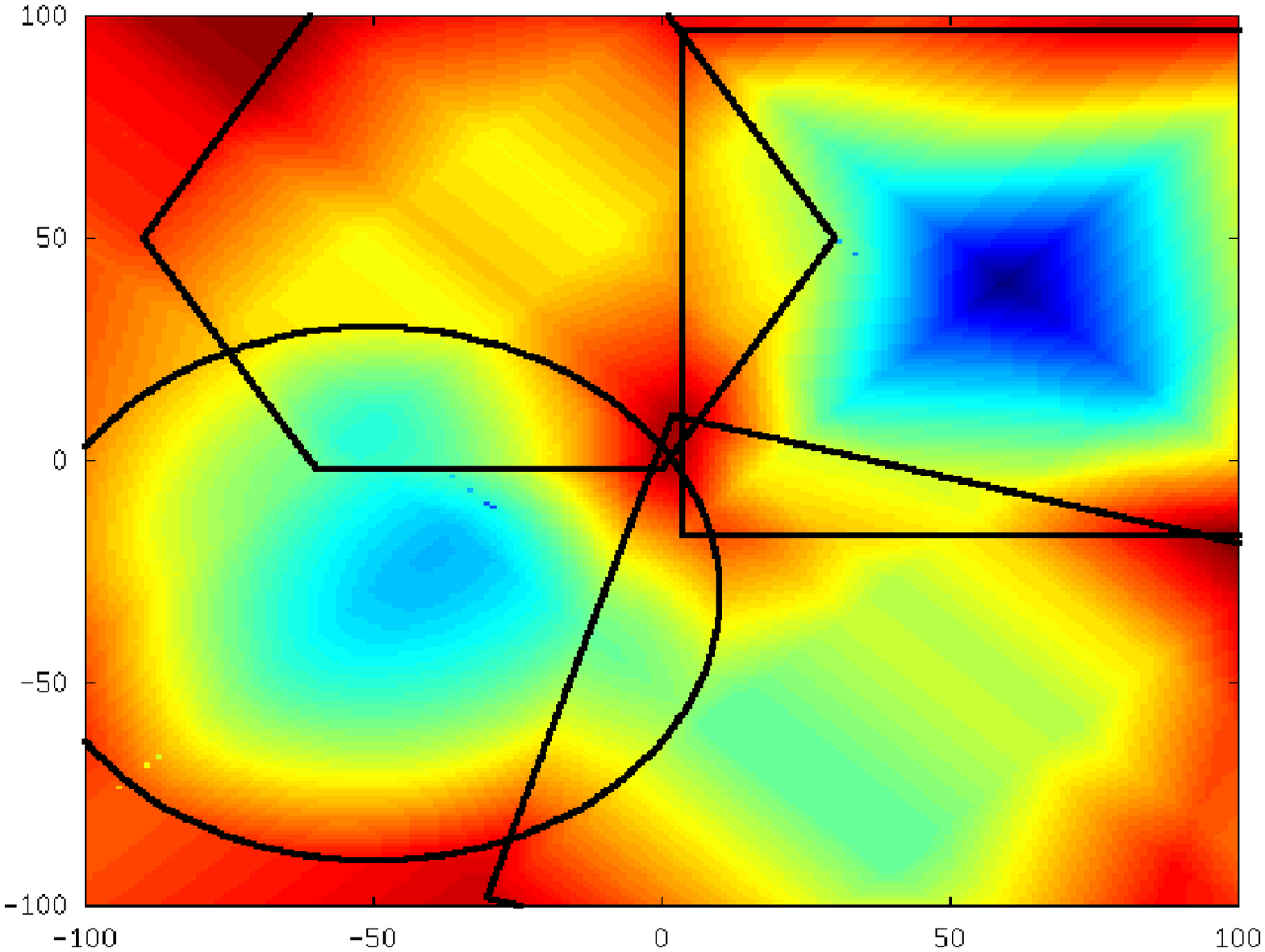}
\includegraphics[width=1.8in,height=1.8in]{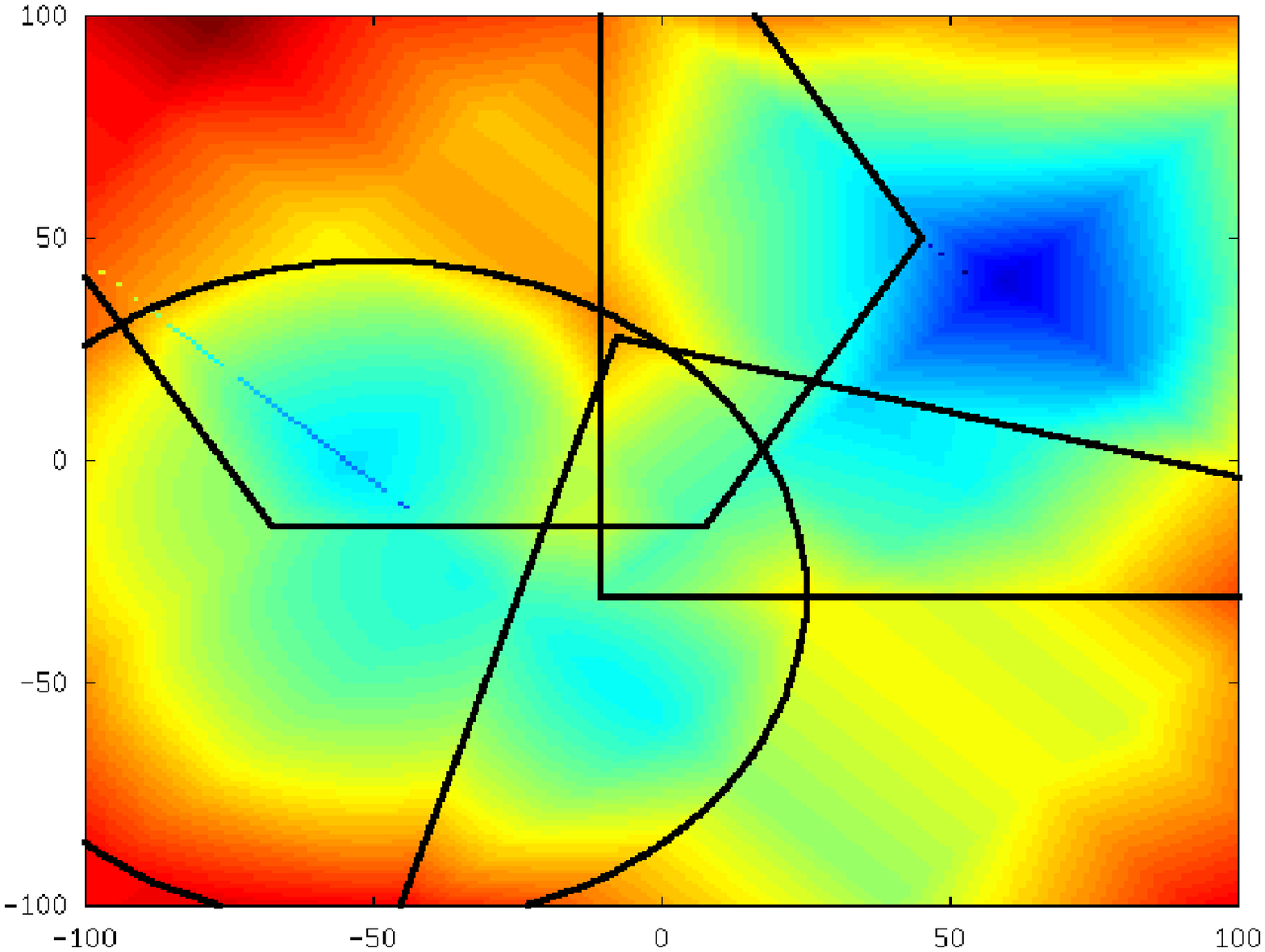}
\caption{Square ($\ell_\infty$) Euler-Bessel transforms of a collection of interfering target supports reveal the location of the square target.}
\label{fig:shapefinder2}
\end{center}
\end{figure}

\begin{figure}
\begin{center}
\includegraphics[width=1.8in,height=1.8in]{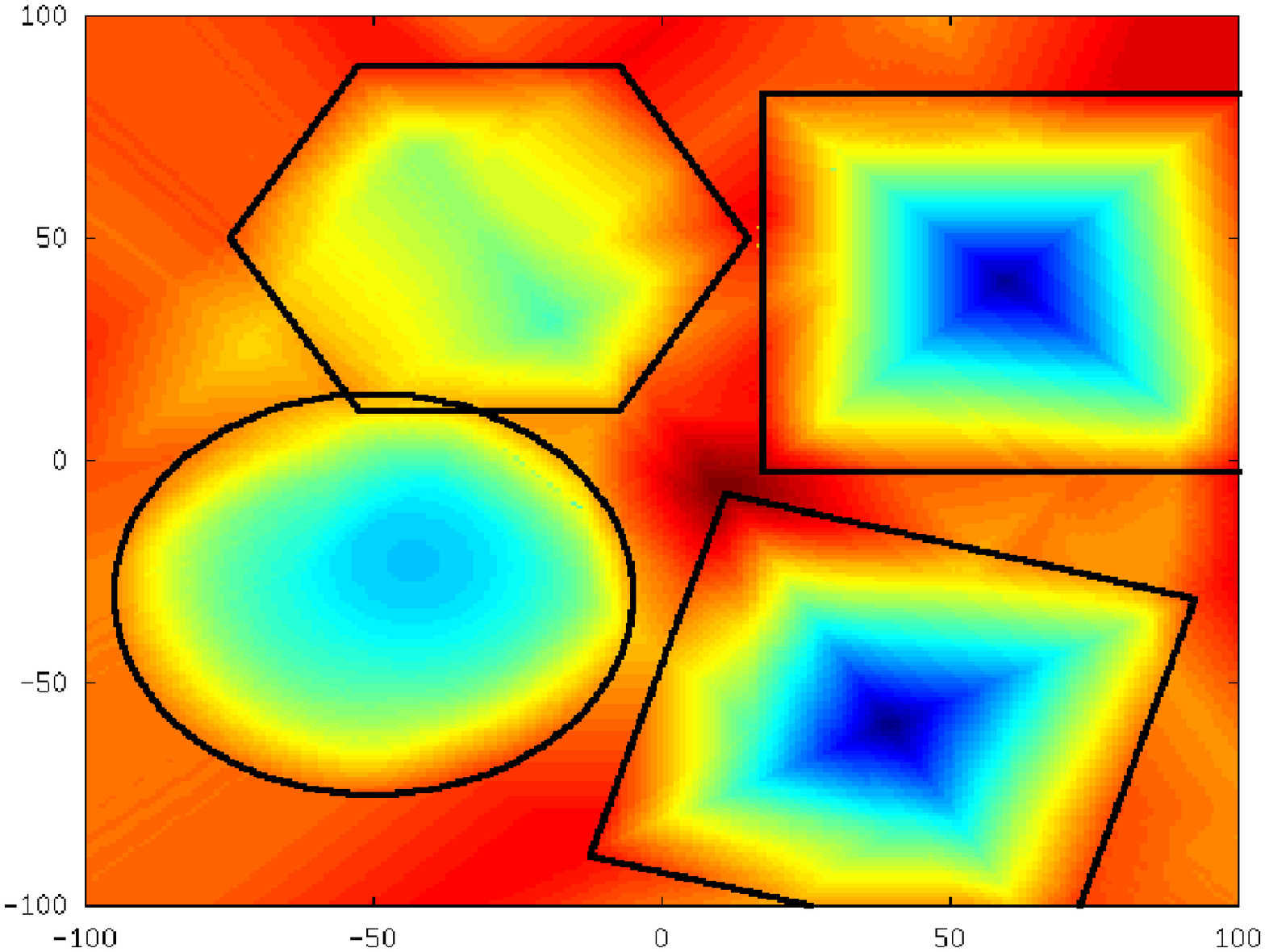}
\includegraphics[width=1.8in,height=1.8in]{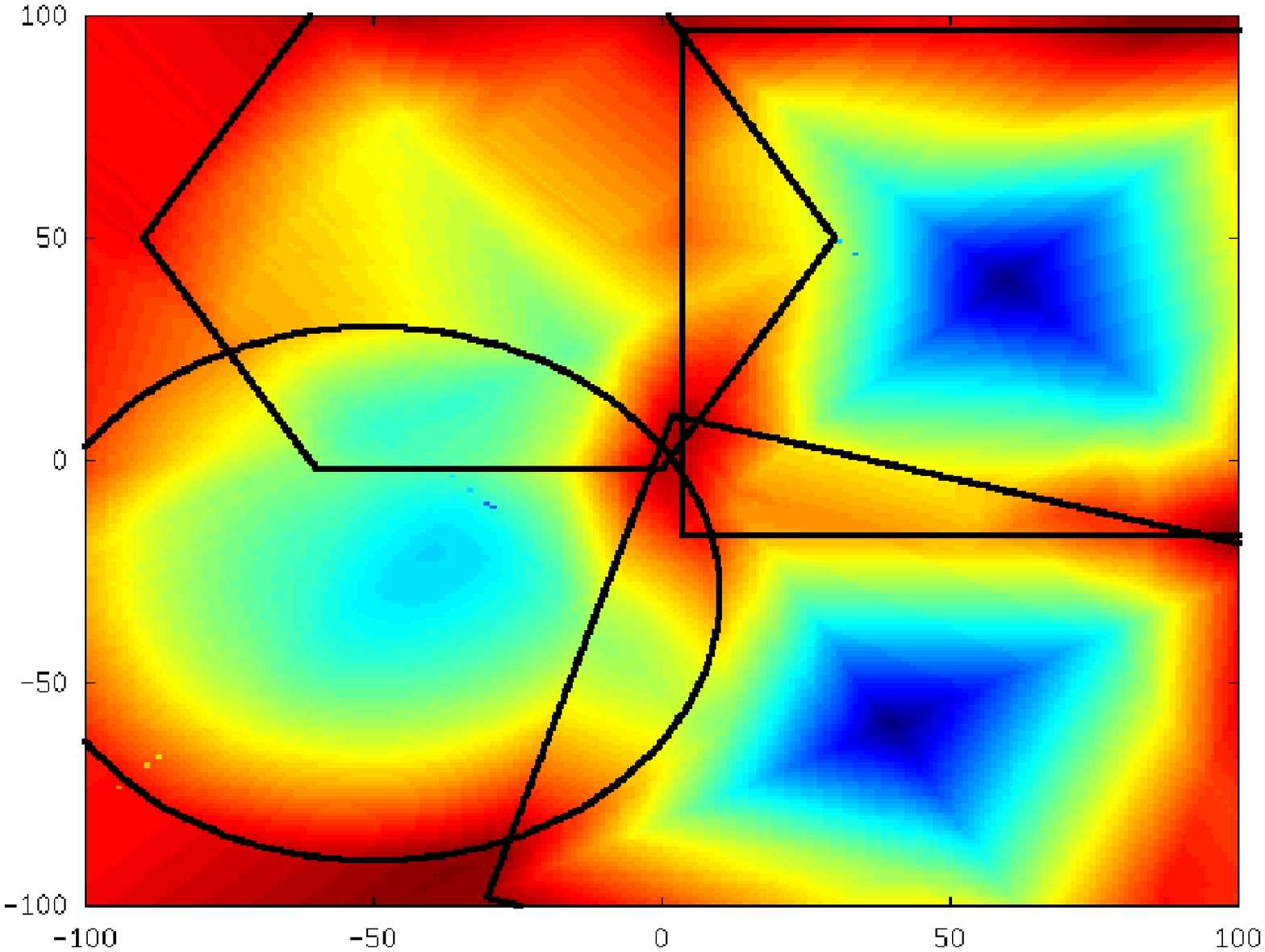}
\includegraphics[width=1.8in,height=1.8in]{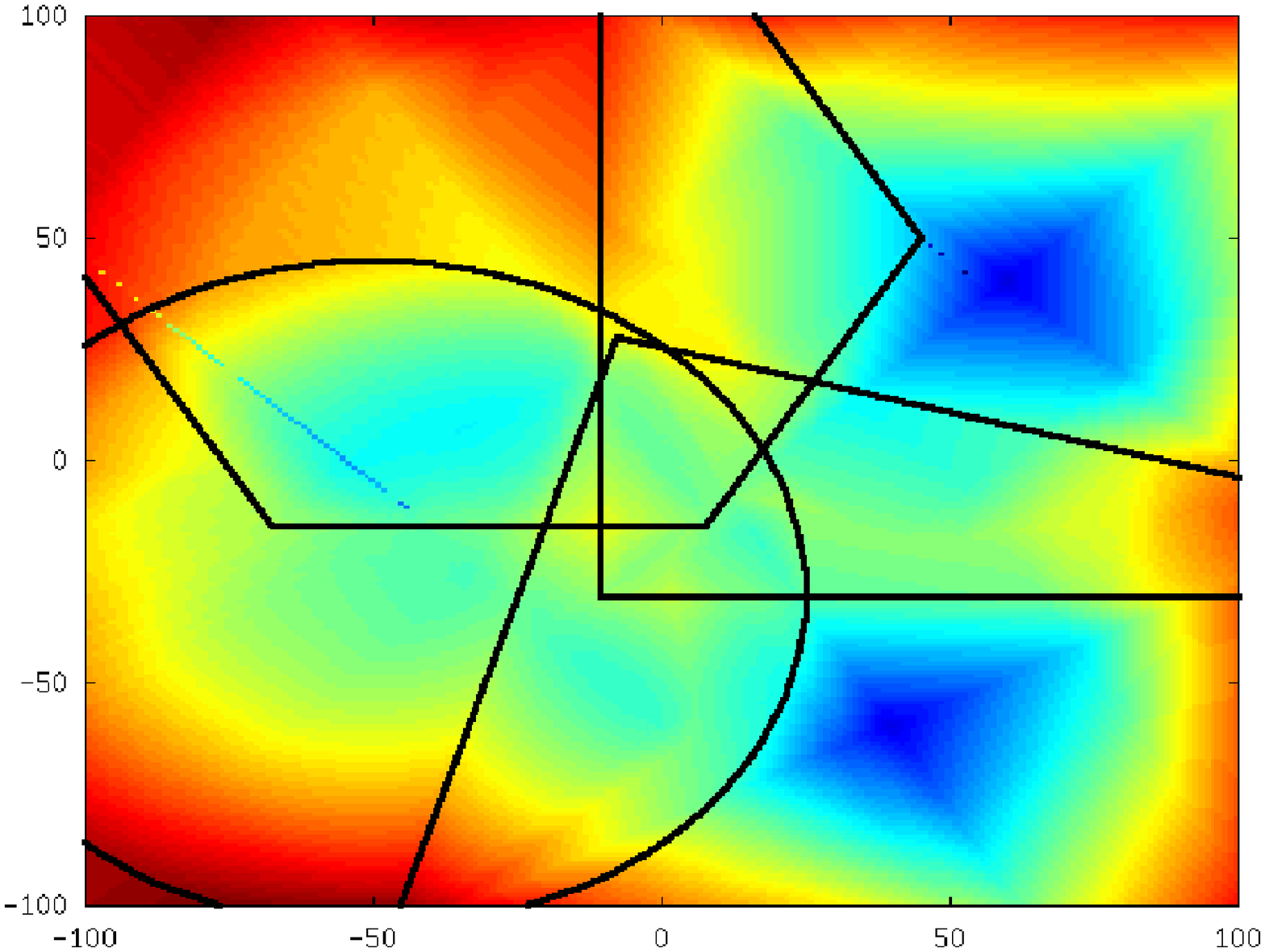}
\caption{SVA (rotated $\ell_\infty$) Euler-Bessel transforms of a collection of interfering target supports reveal the location of the square targets, independent of orientation.}
\label{fig:shapefinder3}
\end{center}
\end{figure}

Three variants of the Euler-Bessel transform were applied to the functions: an $\ell_2$ transform, an $\ell_\infty$ transform , and an SVA (rotated $\ell_\infty$) transform. The deep local minima of each transform correspond to the likely centers for the selected target shapes, in accordance with Proposition \ref{prop:monotone}.  For instance, with the $\ell_2$ transform, although there are minima at the center of the two squares, they are not as deep as the minimum at the center of the disk.  Similarly, the square transform has its deepest minimum at the center of the axis-aligned square, and the SVA transform detects both squares. Even when the support regions overlap, the Euler-Bessel transforms still successfully discriminate target geometries, though too much overlap generates interference and obscures the target geometries (Figure \ref{fig:shapefinder4}).

\begin{figure}
\begin{center}
\includegraphics[width=1.75in,height=1.75in]{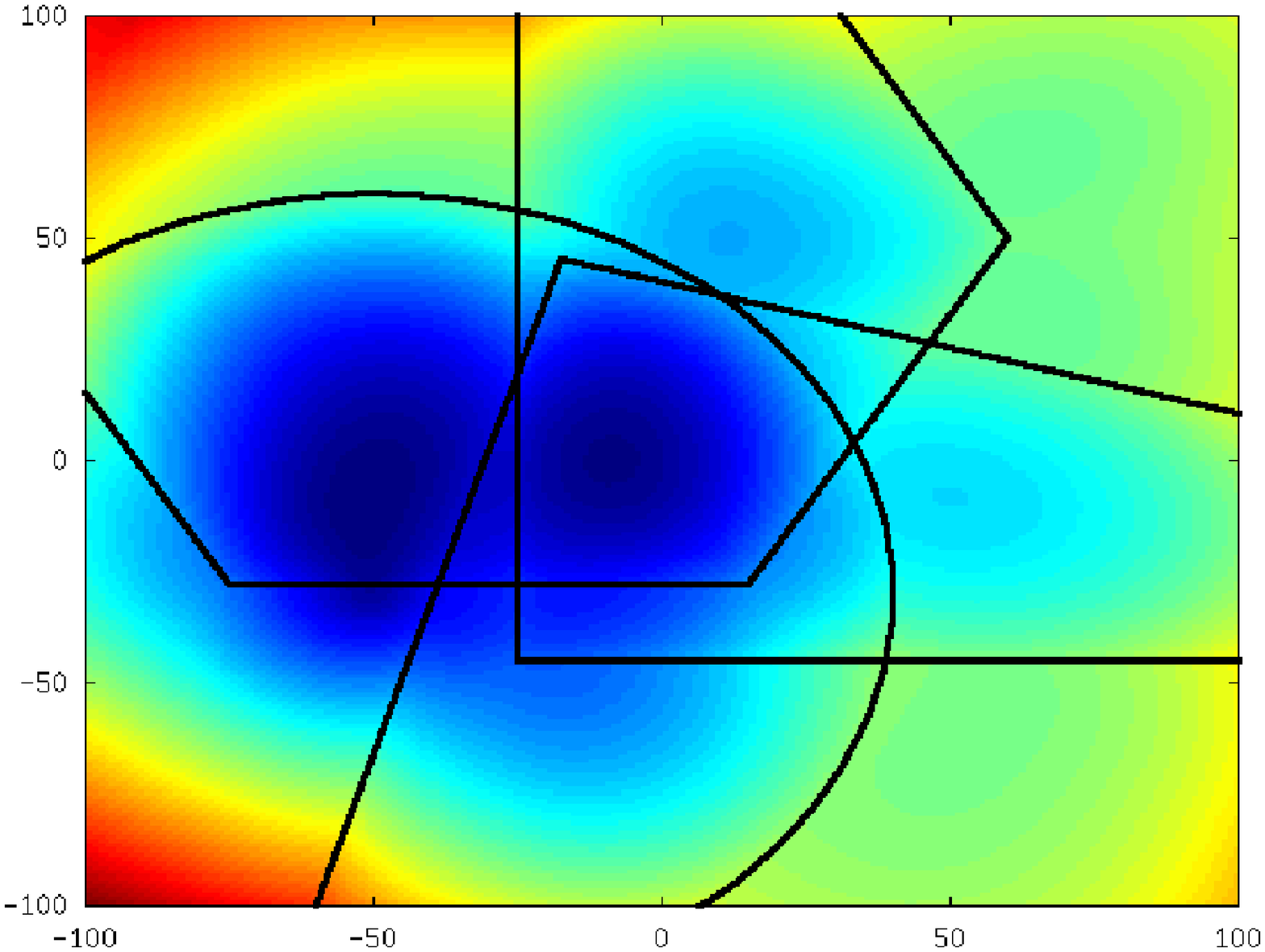}
\includegraphics[width=1.75in,height=1.75in]{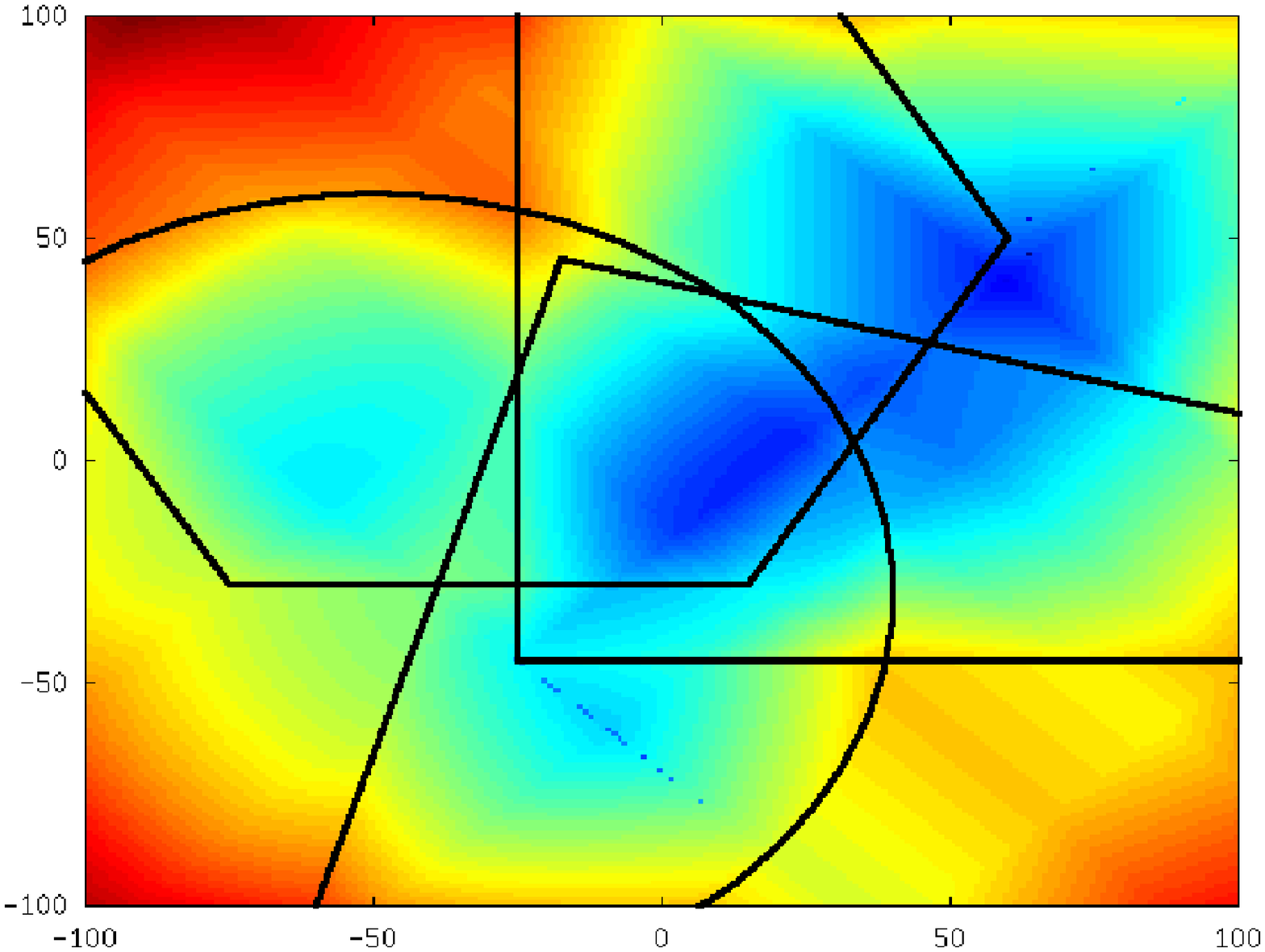}
\includegraphics[width=1.75in,height=1.75in]{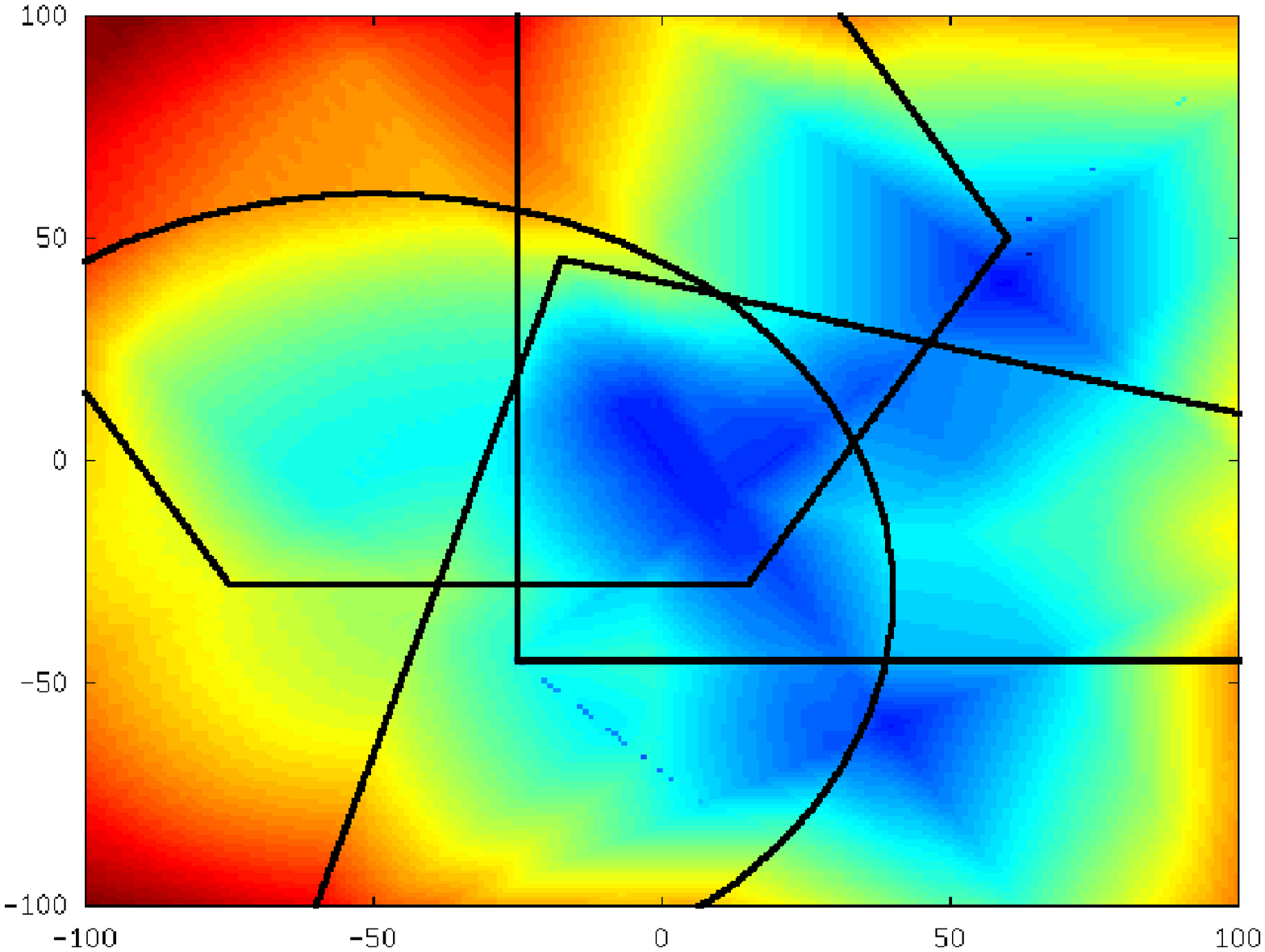}
\caption{When target overlaps are too great, false minima appear in the Euler-Bessel transforms ($\ell_2$, $\ell_\infty$, and SVA $\ell_\infty$)}
\label{fig:shapefinder4}
\end{center}
\end{figure}

Future work will explore the use of SVA in the context of Euler integral transforms.


\section{Concluding remarks}
\label{sec_conc}

\begin{enumerate}
\item
    The index formulae (Theorems \ref{thm:besseldchi} and \ref{thm:fourierdchi}) for the Bessel and Fourier transforms act as a localization of the integral transform and lead to closed-form expressions; this is essential to the computations/simulations in the present paper.
\item
    We have left unaddressed several important issues, such as the existence and nature of inverse and discrete transforms. For inverses to a broad but distinct class of Euler integral transforms, see \cite{Schapira:tom,BGL:radon}.
\item
    Computational issues for discrete Euler-Bessel and Euler-Fourier transforms remain a significant challenge. Though individual Euler integrals can be efficiently approximated on planar networks \cite{BG}, the integral transforms of this paper are, at present, computation-intensive.
\item
    There is no fundamental obstruction to defining $\Fourier$ and $\Bessel$ over $\Def(V)$, so that the transforms act on real-valued (definable) functions. In this case, one would have to distinguish between $\dchifloor$ and $\dchiceil$ versions of the transforms. It is not clear what such transforms measure -- geometrically or topologically -- or which index formulae might persist. The passage to definable inputs means that these transforms will no longer be linear (as neither $\dchifloor$ nor $\dchiceil$ is).
\item
    The use of SVA methods is one of many possible points of contact between the radar signal processing community and Euler integration.
\end{enumerate}

\ack
This work supported by DARPA \# HR0011-07-1-0002 and ONR N000140810668.

\section*{Bibliography}
\bibliographystyle{plain}

\begin{thebibliography}{99}

\bibitem{BG} Y. Baryshnikov and R. Ghrist, ``Target enumeration
    via Euler characteristic integrals,'' {\em SIAM
    J. Appl. Math.}, 70(3), 825-844, 2009.

\bibitem{BG:PNAS} Y. Baryshnikov and R. Ghrist, ``Euler integration for definable functions,''
{\em Proc. National Acad. Sci .}, 107(21), May 25, 9525-9530, 2010.

\bibitem{BGL} Y. Baryshnikov, R. Ghrist, and D. Lipsky ``Inversion of Euler integral transforms with applications to sensor data,''
preprint.

\bibitem{Brocker} L. Br\"ocker, ``Euler integration and Euler
    multiplication,'' {\em Adv. Geom.} 5, 2005, 145--169.

\bibitem{Carrara} W. Carrara, R. Goodman, and R. Majewski, ``Spotlight Synthetic Aperture Radar,'' {\em Signal Processing Algorithms,} Artech House, 1995, 264--278.

\bibitem{Cluckers} R. Cluckers and M. Edmundo, ``Integration of
    positive constructible functions against Euler
    characteristic and dimension,'' {\em J. Pure Appl.
    Algebra}, 208, no. 2, 2006, 691--698.

\bibitem{GM} M. Goresky and R. MacPherson, {\em Stratified
    Morse Theory}, Springer-Verlag, 1988.


\bibitem{KS} M. Kashiwara and P. Schapira, {\em Sheaves on
    Manifolds}, Springer-Verlag, 1994.

\bibitem{KR} D. Klain and G.-C. Rota, {\em Introduction to
    Geometric Probability}, Cambridge University Press, 1997.

\bibitem{Schapira:op} P. Schapira, ``Operations on
    constructible
    functions,'' {\em J. Pure Appl. Algebra} 72, 1991, 83--93.

\bibitem{Schapira:tom} P. Schapira, ``Tomography of
    constructible
    functions,'' in proceedings of {\em 11th Intl. Symp. on
    Applied Algebra, Algebraic Algorithms and Error-Correcting
    Codes}, 1995, 427--435.

\bibitem{SDF} H. Stankwitz, R. Dallaire, and J. Fienup, ``Spatially variant
apodization for sidelobe control in SAR imagery.'' In
IEEE National Radar Conference, 1994, 132--137.

\bibitem{vdD} L. Van den Dries, {\em Tame Topology and
    O-Minimal Structures}, Cambridge University Press, 1998.

\bibitem{Viro} O. Viro, ``Some integral calculus based on Euler
    characteristic,'' Lecture Notes in Math., vol. 1346,
    Springer-Verlag, 1988, 127--138.


\end{thebibliography}

\end{document}